\documentclass{amsart}
\usepackage[latin9]{inputenc}
\usepackage[a4paper]{geometry}
\usepackage{xcolor}
\usepackage{pdfcolmk}
\usepackage{enumitem}
\usepackage{amstext}
\usepackage{amsthm}
\usepackage{amssymb}
\usepackage[all]{xy}
\PassOptionsToPackage{normalem}{ulem}
\usepackage{ulem}

\makeatletter

\providecolor{lyxadded}{rgb}{0,0,1}
\providecolor{lyxdeleted}{rgb}{1,0,0}

\DeclareRobustCommand{\lyxsout}[1]{\ifx\\#1\else\sout{#1}\fi}

\numberwithin{equation}{section}
\numberwithin{figure}{section}
\theoremstyle{plain}
\newtheorem{thm}{\protect\theoremname}[section]
  \theoremstyle{remark}
  \newtheorem*{rem*}{\protect\remarkname}
  \theoremstyle{plain}
  \newtheorem{lem}[thm]{\protect\lemmaname}
  \theoremstyle{plain}
  \newtheorem{prop}[thm]{\protect\propositionname}
  \theoremstyle{plain}
  \newtheorem{cor}[thm]{\protect\corollaryname}
  \theoremstyle{definition}
  \newtheorem{defn}[thm]{\protect\definitionname}

\usepackage{amsthm}\usepackage{bm}
\usepackage{verbatim}
\usepackage{mathrsfs}

\makeatother

  \providecommand{\corollaryname}{Corollary}
  \providecommand{\definitionname}{Definition}
  \providecommand{\lemmaname}{Lemma}
  \providecommand{\propositionname}{Proposition}
  \providecommand{\remarkname}{Remark}
\providecommand{\theoremname}{Theorem}

\begin{document}
\global\long\def\B{\mathcal{B}}

\global\long\def\R{\mathbb{R}}

\global\long\def\S{\mathbb{S}}

\global\long\def\Q{\mathbb{Q}}

\global\long\def\Z{\mathbb{Z}}

\global\long\def\N{\mathbb{N}}

\global\long\def\C{\mathbb{C}}

\global\long\def\Z{\mathbb{Z}}

\global\long\def\D{\mathcal{D}}

\global\long\def\U{\mathcal{U}}

\global\long\def\P{\mathbb{P}}

\global\long\def\G{\mathbb{G}}

\global\long\def\Y{\mathcal{Y}}

\global\long\def\mult{\mathcal{M}}

\global\long\def\ef{\varphi}

\global\long\def\primlat{\mathcal{\Z}_{\text{prim}}^{k,d}}

\global\long\def\primlatof#1{\mathcal{\Z}_{\text{prim}}^{#1,d}}

\global\long\def\zprim{\mathcal{\Z}_{\text{prim}}^{d}}

\global\long\def\gr{\text{Gr}}

\title{on the distribution of primitive subgroups of $\text{\ensuremath{\Z}}^{d}$
of large covolume}

\author{Michael Bersudsky}

\address{Department of mathematics, Technion, Haifa, Israel.}

\email{bersudsky87@gmail.com}
\begin{abstract}
We prove existence and compute the limiting distribution of the image
of rank-$\left(d-1\right)$ primitive subgroups of $\Z^{d}$ of large
covolume in the space $X_{d-1,d}$ of homothety classes of rank-$\left(d-1\right)$
discrete subgroups of $\R^{d}$. This extends a theorem of Aka, Einsiedler
and Shapira.
\end{abstract}

\maketitle

\section{Introduction}

\thispagestyle{empty}

\let\thefootnote\relax\footnotetext{This project has received funding from the European Research Council  (ERC) under the European Union's Horizon 2020 research and innovation programme (grant agreement No. 754475). The author also acknowledges the support of ISF grant 871/17.} 
Let $\bar{X}_{d-1,d}$ be the space of rank-$\left(d-1\right)$ discrete
subgroups of $\R^{d}$ and let $X_{d-1,d}$ be the space of their
homothety classes, that is
\[
X_{d-1,d}\overset{\text{def}}{=}\bar{X}_{d-1,d}\diagup\sim,
\]
where,
\[
\text{\ensuremath{\Lambda}}\sim\alpha\cdot\Lambda,\ \alpha\in\R^{\times}.
\]
For $\Lambda\in\bar{X}_{d-1,d}$, we will denote by $\left[\Lambda\right]$
its image in $X_{d-1,d}$, namely

\[
\left[\Lambda\right]\overset{\text{def}}{=}\left\{ \alpha\cdot\Lambda\mid\alpha\in\R^{\times}\right\} .
\]
 The\emph{ covolume} function, denoted by 
\[
\text{cov}:\ \bar{X}_{d-1,d}\to\R_{+},
\]
is the function that assigns to each $\Lambda\in\bar{X}_{d-1,d}$
the $\left(d-1\right)$-volume of a fundamental parallelotope. It
is explicitly given by
\begin{equation}
\text{cov}(\Lambda)=\sqrt{\det b^{t}b},\ \Lambda\in X_{d-1,d},\label{eq:covol def}
\end{equation}
where $b\in M_{d\times d-1}(\R)$ is a matrix whose columns form a
$\Z$-basis for $\Lambda$. Let $\primlatof{d-1}\subseteq\bar{X}_{d-1,d}$
be the set of all subgroups that arise as the intersection of  $\Z^{d}$
with a rational hyperplane. Of particular interest to us are the subsets
of $X_{d-1,d}$ defined by
\[
\primlatof{d-1}(T)\overset{\text{def}}{=}\left\{ \left[\Lambda\right]\in X_{d-1,d}\mid\Lambda\in\primlatof{d-1},\ \text{cov}(\Lambda)=T\right\} ,\ T\in\R^{\times}.
\]
The sets $\primlatof{d-1}(T),$ $T\in\R^{\times}$ are finite (see
e.g. Lemma \ref{lem:bijection orthogonal vectors}) and we denote
by $\mu_{T},$ $T\in\R^{\times}$, the uniform probability counting
measures on them. In this note we prove that certain subsequences
of $\left\{ \mu_{T}\right\} $ converge to a probability measure which
we denote by $\mu_{\text{polar}}$. The measure $\mu_{\text{polar}}$
may be described through disintegration (see e.g. \cite{sergant_shapira})
as follows. Let $\gr_{d-1}(\R^{d})$ be the space of hyperplanes in
$\R^{d}$, and for $\mathcal{P}\in\gr_{d-1}(\R^{d})$ let
\begin{equation}
X_{\mathcal{P}}\overset{\text{def}}{=}\left\{ \left[\Lambda\right]\in X_{d-1,d}\mid\Lambda\subseteq\mathcal{P}\right\} .\label{eq:X_p}
\end{equation}
A choice of a linear isomorphism between $\mathcal{P}$ and $\R^{d-1}$
gives an identification of $X_{\mathcal{P}}$ with $X_{d-1}=\text{PGL}{}_{d-1}(\R)\diagup\text{PGL}{}_{d-1}(\Z)$.
Through this identification, the $\text{PGL}{}_{d-1}(\R)$-invariant
probability measure $\mu_{X_{d-1}}$ may be pushed to a measure on
$X_{\mathcal{P}}.$ It turns out that the latter push-forward to $X_{\mathcal{P}}$
is independent of the chosen isomorphism (see e.g. part 3 of Lemma
\ref{lem:horizontal lattices identification with =00005BX_d-1=00005D}),
hence this process defines a measure $\mu_{\mathcal{P}}$ on $X_{\mathcal{P}}$.
Let $\mu_{\gr_{d-1}(\R^{d})}$ be the $\text{SO}{}_{d}(\R)$-invariant
probability measure on $\gr_{d-1}(\R^{d})$. We define
\[
\mu_{\text{polar}}\overset{\text{def}}{=}\int_{\gr_{d-1}(\R^{d})}\mu_{\mathcal{P}}\ d\mu_{\gr_{d-1}(\R^{d})}(\mathcal{P}).
\]
For a prime $p$, let
\begin{equation}
\mathbb{D}(p)\overset{\text{def}}{=}\left\{ m\in\N\mid p\nmid m\right\} .\label{eq:D(P)}
\end{equation}
We prove 
\begin{thm}
\label{thm:maintheorem} The convergence 
\[
\mu_{T_{j}}\overset{\text{weak * }}{\longrightarrow}\mu_{\text{polar}},
\]
holds for:
\begin{enumerate}
\item $d=4$, and $\left\{ T_{j}^{2}\right\} \subseteq\mathbb{D}(p)/8\N$,
for any odd prime $p$.
\item $d=5$, and $\left\{ T_{j}^{2}\right\} \subseteq\mathbb{D}(p)$, for
any odd prime $p$.
\item $d>5$, and $\left\{ T_{j}^{2}\right\} \subseteq\mathbb{N}$.
\end{enumerate}
\end{thm}

\begin{rem*}
It should be possible to prove a version of Theorem \ref{thm:maintheorem}
for $d=3$ by relying on the work of \cite{AES3d}. Also, it seems
that the unnecessary congruence conditions in dimensions $4$ and
$5$ can be removed and an effective estimate on the convergence can
be obtained by exploiting a theorem of Einsiedler, R{\"u}hr and Wirth
found in \cite{Effective_aes_Ruhr}. In order to do so, one should
replace Theorem \ref{thm:AESgrids thm} (of \cite{AESgrids}) stated
in this note, with the corresponding theorems of \cite{AES3d} and
\cite{Effective_aes_Ruhr} and go along the lines of sections \ref{sec:The-p-adic-factory}
and \ref{sec:proof of the equidisitrbution}. 
\end{rem*}

\subsection{Background for Theorem \ref{thm:maintheorem} }

W. Schmidt in \cite{Schmidt1998,Schmidt2015} computed the distribution
of the homothety classes the $\left(d-1\right)$-integral subgroups
through the filtration 
\[
\left\{ \left[\Lambda\right]\in X_{d-1,d}\mid\Lambda\in\primlatof{d-1},\ \ \text{cov}(\Lambda)\leq T\right\} ,\ \text{as }T\to\infty.
\]
Hence, Theorem \ref{thm:maintheorem} should be viewed as a sparse
version of Schmidt's result. Later, Aka, Einsiedler and Shapira in
\cite{AESgrids,AES3d} computed the limiting distribution of the image
of the sets $\primlatof{d-1}(T)$ in the space $\gr_{d-1}(\R^{d})\times\text{O}_{d}(\R)\backslash X_{d-1,d}$.
Since there is a natural surjection $\pi:X_{d-1,d}\to\gr_{d-1}(\R^{d})\times\text{O}_{d}(\R)\backslash X_{d-1,d}$,
Theorem \ref{thm:maintheorem} implies the main result of \cite{AESgrids}.
We also note that the type of problem considered here may be viewed
as a natural generalization of the problem considered by Y. Linnik
regarding the equidistribution of the projection to the unit sphere
of primitive integer vectors on large spheres (see \cite{Lin68},
and also \cite{elenberg_michel_venaktesh} for a modern review). 

\subsection{Organization of the note and proof ideas}

We provide a novel interpretation of $X_{d-1,d}$ as a double coset
space (see Proposition \ref{prop:Tthe polar coordinates}) which allows
to use the methods and results of \cite{AESgrids} in order to prove
Theorem \ref{thm:maintheorem}. In an overview, the method of \cite{AESgrids}
allows to interpret the sets $\primlatof{d-1}(T)$ as compact orbits
in an S-arithmetic space and to relate their natural measure to the
measures $\mu_{T}$. A key theorem of \cite{AESgrids} states that
those orbits equidistribute, which eventually allows to deduce Theorem
\ref{thm:maintheorem}. The organization of the note is as follows. 
\begin{itemize}
\item In Section \ref{sec:Polar-coordinates-on} we describe $X_{d-1,d}$
as a coset space, and as a double coset space. We also discuss the
measure $\mu_{\text{polar }}$ in detail. 
\item In Section \ref{sec:The-p-adic-factory} we discuss the method that
is used to ``generate'' elements of $\primlatof{d-1}$ by the p-adics. 
\item In Section \ref{sec:proof of the equidisitrbution} we discuss the
resulting measures and conclude the proof.
\end{itemize}

\section{\label{sec:Polar-coordinates-on} $X_{d-1,d}$ as a Homogenous space
and its polar coordinates}

\subsection{The transitive action of $\text{SL}{}_{d}(\protect\R)$}

The group $\text{SL}{}_{d}(\R)$ acts from the left on $X_{d-1,d}$
by 
\[
g\cdot[\Lambda]=[g\Lambda],\ g\in\text{SL}{}_{d}(\R),\ [\Lambda]\in X_{d-1,d},
\]
where 
\[
g\Lambda=\left\{ gv\mid v\in\Lambda\right\} .
\]
Let $e_{i},$ $i\in\{1,..,d\}$ be the standard basis vectors of $\R^{d}$
and note that for $g\in\text{SL}{}_{d}(\R)$, the set $\left\{ ge_{1},..,ge_{d-1}\right\} $
consists of the first $\left(d-1\right)$ columns of $g$. Since basis
vectors of any $\Lambda\in X_{d-1,d}$ can be put to be the first
$\left(d-1\right)$ columns of some $\text{SL}{}_{d}(\R)$ matrix,
we deduce that the $\text{SL}{}_{d}(\R)$ orbit of 
\[
x_{0}\overset{\text{def}}{=}\left[\text{span }_{\Z}\{e_{1},..,e_{d-1}\}\right]
\]
equals to $X_{d-1,d}$. A computation shows that 
\[
Q_{d-1,d}\overset{\text{def}}{=}\left\{ \left(\begin{array}{cc}
\lambda\gamma & *\\
0_{1\times d} & 1/\det(\lambda\gamma)
\end{array}\right)\mid\lambda\in\R^{\times},\gamma\in\text{GL}{}_{d-1}(\Z)\right\} 
\]
is the stabilizer of $x_{0}$. Therefore, we get the identification
\[
X_{d-1,d}=\text{SL}{}_{d}(\R)\diagup Q_{d-1,d}.
\]
\begin{rem*}
In terms of the coset space, the collection $\left\{ \left[\Lambda\right]\in X_{d-1,d}\mid\Lambda\in\Z_{\text{prim}}^{d-1,d}\right\} $
is identified with the orbit $\text{SL}{}_{d}(\Z)x_{0}$.
\end{rem*}

\subsubsection{The measure $\mu_{\text{polar }}$}

Let $P_{d-1,d}$ be the parabolic group,
\[
P_{d-1,d}\overset{\text{def}}{=}\left\{ \left(\begin{array}{cc}
m & *\\
0_{1\times d} & 1/\det m
\end{array}\right)\mid m\in\text{GL}{}_{d-1}(\R)\right\} .
\]
\begin{lem}
\label{lem:horizontal lattices identification with =00005BX_d-1=00005D}
Let \emph{$g\in\text{SL}{}_{d}(\R)$}, and let $\mathcal{P}=\text{span}_{\R}\{ge_{1},..,ge_{d-1}\}$.
Then:
\begin{enumerate}
\item \label{enu:lattices in hyperplane 1} The subset $X_{\mathcal{P}}$
\emph{(}see \eqref{eq:X_p}\emph{)} is identified with $gP_{d-1,d}\diagup Q_{d-1,d}$.
\item \label{enu:lattices in hyperplane 2}The map\emph{
\[
\ef_{g}:\text{PGL}{}_{d-1}(\R)\diagup\text{PGL}{}_{d-1}(\Z)\to X_{\mathcal{P}},
\]
}which sends\emph{
\[
\R^{\times}m\cdot\text{PGL}{}_{d-1}(\Z)\mapsto g\left(\begin{array}{cc}
m & 0_{d-1\times1}\\
0_{1\times d-1} & 1
\end{array}\right)Q_{d-1,d},
\]
}is a homeomorphism.
\item \label{enu:lattices in hyperplane 3}Assume that $gP_{d-1,d}=g'P_{d-1,d}$.
Then,
\begin{equation}
\left(\ef_{g}\right)_{*}\mu_{X_{d-1}}=\left(\ef_{g'}\right)_{*}\mu_{X_{d-1}},\label{eq:pushforwards equal}
\end{equation}
where $\mu_{X_{d-1}}$ is the \emph{$\text{PGL}{}_{d-1}(\R)$}-invariant
probability measure on $X_{d-1}$. 
\end{enumerate}
\end{lem}

\begin{proof}
We observe that 
\[
\left\{ g\in\text{SL}_{d}(\R)\mid\mathcal{P}=\text{span}_{\R}\{ge_{1},..,ge_{d-1}\}\right\} =gP_{d-1,d},
\]
hence \eqref{enu:lattices in hyperplane 1} follows. Next, in order
prove \eqref{enu:lattices in hyperplane 2}, we note that the map
\[
\phi_{1}:\text{PGL}{}_{d-1}(\R)\diagup\text{PGL}{}_{d-1}(\Z)\to P_{d-1,d}\diagup Q_{d-1,d},
\]
 that sends,
\[
\R^{\times}m\cdot\text{PGL}{}_{d-1}(\Z)\mapsto\left(\begin{array}{cc}
m & 0_{d-1\times1}\\
0_{1\times d-1} & 1/\det m
\end{array}\right)Q_{d-1,d},\ \ m\in\text{GL}{}_{d-1}(\R),
\]
 is a homeomorphism, and the map
\[
\phi_{2}:P_{d-1,d}\diagup Q_{d-1,d}\to gP_{d-1,d}\diagup Q_{d-1,d},
\]
defined by multiplication from the left by $g,$ is also a homeomorphism.
Hence $\phi_{2}\circ\phi_{1}$ is a homeomorphism, which proves \eqref{enu:lattices in hyperplane 2}.
Finally, we prove \eqref{enu:lattices in hyperplane 3}. We let $g,g'\in\text{SL}{}_{d}(\R)$
be such that 
\[
g'=gp,
\]
for some $p\in P_{d-1,d}$, where $p=\left(\begin{array}{cc}
m_{p} & v\\
0_{1\times d} & 1/\det m_{p}
\end{array}\right)$. Then, a short calculation shows that 
\begin{equation}
\ef_{g'}\left(\R^{\times}m\text{PGL}{}_{d-1}(\Z)\right)=\ef_{g}\left(\R^{\times}m_{p}m\text{PGL}{}_{d-1}(\Z)\right),\label{eq:ef_g and ef_g_tilde}
\end{equation}
hence, since $\mu_{X_{d-1}}$ is $\text{PGL}{}_{d-1}(\R)$-invariant,
we obtain \eqref{eq:pushforwards equal}. 
\end{proof}
\begin{rem*}
In the rest, we shall abuse notation and denote the measure $\left(\ef_{id}\right)_{*}\mu_{X_{d-1}}$
on $P_{d-1,d}\diagup Q_{d-1,d}$ by $\mu_{X_{d-1}}$.
\end{rem*}
We denote 
\[
K_{d-1,d}^{\pm}\overset{\text{def}}{=}P_{d-1,d}\cap\text{SO}{}_{d}(\R),
\]
which is isomorphic to $\text{O}_{d-1}(\R),$ and identify $Gr_{d-1}(\R^{d})$
with $K_{d-1,d}^{\pm}\diagdown\text{SO}{}_{d}(\R)$ via the map
\[
K_{d-1,d}^{\pm}\rho\mapsto\text{span}_{\R}\{\rho^{-1}e_{1},..,\rho^{-1}e_{d-1}\}.
\]
\begin{rem*}
To ease the notation, we will omit in the following the indices $d-1,d$
from $P_{d-1,d}$, $Q_{d-1,d}$ and $K_{d-1,d}^{\pm}$. 
\end{rem*}
Let $\mu_{Gr_{d-1}(\R^{d})}$ be the $\text{SO}{}_{d}(\R)$-invariant
probability measure and for $f\in C_{c}(X_{d-1,d})$ we define $\hat{f}\in C_{c}(Gr_{d-1}(\R^{d}))$
by
\[
\hat{f}(K^{\pm}\rho)\overset{\text{def}}{=}\left(\ef_{\rho^{-1}}\right)_{*}\mu_{X_{d-1}}(f),
\]
which is well defined by part 3 of Lemma \ref{lem:horizontal lattices identification with =00005BX_d-1=00005D}.
Then, the measure $\mu_{\text{polar }}$ is given by 
\begin{equation}
\mu_{\text{polar }}(f)\overset{\text{def}}{=}\int_{\text{Gr}_{d-1}(\R^{d})}\hat{f}(K^{\pm}\rho)\ d\mu_{Gr_{d-1}(\R^{d})}(K^{\pm}\rho).\label{eq:mu polar section 2}
\end{equation}

Note that the description \eqref{eq:mu polar section 2} of $\mu_{\text{polar }}$
yields the same measure defined in the introduction, although stated
slightly differently. We chose this description since it suits well
to the proof of Lemma \ref{lem:mult push nu polar to mu polar}.

\subsection{Alternative description of $X_{d-1,d}$ via polar coordinates}

Here we shall give a description of the elements of $X_{d-1,d}$ by
their orientation and by their shape, hence the name of \emph{polar
coordinates}. Those coordinates will serve as a bootstrap to the technique
of \cite{AESgrids}.

\subsubsection{The multiplication map}

Let $\Delta K^{\pm}$ be the diagonal embedding in $\text{SO}{}_{d}(\R)\times P$,
which is defined by
\[
\Delta K^{\pm}\overset{\text{def}}{=}\left\{ (k,k)\mid k\in K_{d-1,d}^{\pm}\right\} \leq\text{SO}{}_{d}(\R)\times P.
\]

The following double coset space,
\[
X_{d-1,d}^{\text{polar}}\overset{\text{def}}{=}\Delta K^{\pm}\diagdown\left(\text{SO}{}_{d}(\R)\times P\diagup Q\right),
\]
will be shown to be homeomorphic to $X_{d-1,d}$. Consider the map
\[
\mult\ :\ X_{d-1,d}^{\text{polar}}\to X_{d-1,d},
\]
defined by 
\[
\mult\left(\Delta K^{\pm}(\rho,\eta Q)\right)=\rho^{-1}\eta Q.
\]
It is well defined since if $(\rho',\eta'Q)=(k\rho,k\eta Q),$ then
$\rho'^{-1}\eta'Q=\rho^{-1}\eta Q$. 
\begin{prop}
\label{prop:Tthe polar coordinates}The map $\mult$ is a homeomorphism. 
\end{prop}

\begin{proof}
To prove injectivity, we assume that 
\[
\mult\left(\Delta K^{\pm}(\rho_{1},\eta_{1}Q)\right)=\mult\left(\Delta K^{\pm}(\rho_{2},\eta_{2}Q)\right),
\]
which is equivalent to that
\[
\rho_{1}^{-1}\eta_{1}q=\rho_{2}^{-1}\eta_{2},
\]
for some $q\in Q$. Then, 
\[
SO_{d}(\R)\ni\rho_{2}\rho_{1}^{-1}=\eta_{2}q^{-1}\eta_{1}^{-1}\in P,
\]
hence there is a $k\in K^{\pm}$ such that 
\[
\rho_{2}\rho_{1}^{-1}=\eta_{2}q^{-1}\eta_{1}^{-1}=k,
\]
which in turn implies
\[
\Delta K^{\pm}(\rho_{1},\eta_{1}Q)=\Delta K^{\pm}(\rho_{2},\eta_{2}Q).
\]
To prove continuity of $\mult$, we consider the following commuting
diagram
\[
\xymatrix{\text{SO}{}_{d}(\R)\times P\ar[d]\ar[r] & \text{SL}{}_{d}(\R)\ar[d]\\
X_{d-1,d}^{\text{polar}}\ar[r]_{\ \ \mult} & X_{d-1,d}
}
\]
where the vertical maps are the natural projections, and the horizontal
upper arrow sends
\[
(\rho,\eta)\mapsto\rho^{-1}\eta.
\]
Note that the resulting map from $\text{SO}{}_{d}(\R)\times P$ to
$X_{d-1,d}$ is a composition of continuous maps, hence is continuous.
Therefore, by the universal property of the quotient space, $\mult$
is continuous. Next, we compute the inverse of $\mult$ and show it
is continuous. Let $A\leq\text{SL}{}_{d}(\R)$ be the diagonal subgroup
with positive entries, and $N\leq\text{SL}{}_{d}(\R)$ be the group
of upper triangular unipotent matrices. The map (Iwasawa decomposition)
\[
\psi:\text{SO}{}_{d}(\R)\times A\times N\to\text{SL}{}_{d}(\R),
\]
given by
\[
\psi(\rho,a,n)=\rho an,
\]
is a homeomorphism (see e.g. \cite{Bekka_mayer}, Chapter 5). Consider
the following commuting diagram,
\[
\xymatrix{\text{SO}{}_{d}(\R)\times P\ar[d] & \text{SO}{}_{d}(\R)\times A\times N\ar_{p\ }[l] & \text{SL}{}_{d}(\R)\ar[l]_{\ \ \ \ \ \ \ \ \ \psi^{-1}}\ar[d]\\
X_{d-1,d}^{\text{polar}} &  & X_{d-1,d}\ar[ll]
}
\]
where the map $p$ is defined by 
\[
p(\rho,a,n)\overset{\text{def}}{=}\left(\rho^{-1},an\right),
\]
and the horizontal maps are the natural projections. The map corresponding
to the lower horizontal arrow sends
\[
\rho anQ\mapsto\Delta K^{\pm}\left(\rho^{-1},anQ\right),
\]
which is clearly an inverse for $\mult$. Since the resulting map
from $\text{SL}_{d}(\R)$ to $X_{d-1,d}^{\text{polar}}$ is a composition
of continuous maps, we get that it is continuous, hence by the universal
property of the quotient space, $\mult^{-1}$ is continuous.
\end{proof}

\subsubsection{The measure $\mu_{\text{polar }}$ through the polar coordinates}

Consider the map 
\[
q_{\Delta K^{\pm}}:\text{SO}{}_{d}(\R)\times P\diagup Q\to X_{d-1,d}^{\text{polar}},
\]
that divides from the left by $\Delta K^{\pm}$. We define 
\begin{equation}
\nu_{\text{polar}}\overset{\text{def}}{=}\left(q_{\Delta K^{\pm}}\right)_{*}\mu_{\text{SO}{}_{d}(\R)}\otimes\mu_{X_{d-1}}.\label{eq:nu_polar definition}
\end{equation}
 
\begin{lem}
\label{lem:mult push nu polar to mu polar}It holds that $\mult_{*}\nu_{\text{polar}}=\mu_{\text{polar}}$.
\end{lem}

\begin{proof}
First, recall that for $\ef\in C_{c}(\text{SO}{}_{d}(\R)$) ,

\[
\int_{\text{SO}{}_{d}(\R)}\ef\ d\mu_{\text{SO}{}_{d}(\R)}=\int_{\gr_{d-1}(\R^{d})}\left(\int_{K^{\pm}}\ef(k\rho)\ d\mu_{K^{\pm}}(k)\right)d\mu_{\gr_{d-1}(\R^{d})}(K^{\pm}\rho).
\]
Hence for $f\in C_{c}(X_{d-1,d}^{\text{polar}})$,
\[
\nu_{\text{polar}}(f)=\int f(q_{\Delta K^{\pm}}(\rho,\eta Q))d\mu_{X_{d-1}}(\eta Q)d\mu_{\text{SO}{}_{d}(\R)}(\rho)=
\]
\begin{equation}
\int\left(\int f(q_{\Delta K^{\pm}}(k\rho,\eta Q))d\mu_{X_{d-1}}(\eta Q)d\mu_{K^{\pm}}(k)\right)d\mu_{\gr_{d-1}(\R^{d})}(K^{\pm}\rho).\label{eq:integral x_d-1,K_Gr_d-1}
\end{equation}
Note that 
\[
q_{\Delta K^{\pm}}(k\rho,\eta Q)=q_{\Delta K^{\pm}}(\rho,k^{-1}\eta Q),
\]
whence, by \eqref{eq:integral x_d-1,K_Gr_d-1},
\[
\nu_{\text{polar}}(f)=\int\left(\int f(q_{\Delta K^{\pm}}(\rho,k^{-1}\eta Q))d\mu_{X_{d-1}}(\eta Q)\right)d\mu_{K^{\pm}}(k)d\mu_{\gr_{d-1}(\R^{d})}(K^{\pm}\rho).
\]
The measure $\mu_{X_{d-1}}$ is $K^{\pm}$ invariant, so that 
\[
\nu_{\text{polar}}(f)=\int\left(\int f(q_{\Delta K^{\pm}}(\rho,\eta Q))d\mu_{X_{d-1}}(\eta Q)\right)d\mu_{\gr_{d-1}(\R^{d})}(K^{\pm}\rho).
\]
Finally, the push-forward by $\mult$ gives
\[
\mult_{*}\nu_{\text{polar}}(f)=\int\left(\int f(\rho^{-1}\eta Q)d\mu_{X_{d-1}}(\eta Q)\right)d\mu_{\gr_{d-1}(\R^{d})}(K^{\pm}\rho),
\]
where
\[
\int f(\rho^{-1}\eta Q)d\mu_{X_{d-1}}(\eta Q)=\left(\ef_{\rho^{-1}}\right)_{*}\mu_{X_{d-1}}(f).
\]
In view of the definition \eqref{eq:mu polar section 2}, the proof
is now done.
\end{proof}
\begin{rem*}
The whole discussion of this section can be adjusted with no trouble
to the spaces $X_{k,d}$ of homothety classes of of rank-$k$ discrete
subgroups of $\R^{d}$, for $1\leq k<d$. 
\end{rem*}

\section{\label{sec:The-p-adic-factory}The p-adic factory of primitive integral
subgroups }

\subsection{\label{subsec:The-mechanism}The mechanism}

In order to better connect our discussion to the one of \cite{AESgrids},
we shall recall the description of the elements $\Lambda\in\primlatof{d-1}$
with fixed covolume as orthogonal lattices to integer vectors of fixed
norm. Let $\Z_{\text{prim}}^{d}$ be the set of integral primitive
vectors. For a primitive integer vector $v\in\Z_{\text{prim }}^{d}$,
let $v^{\perp}\in Gr_{d-1}(\Q^{d})$ be the orthogonal hyperplane
to $v$. We define the \emph{orthogonal lattice }to $v$ by
\[
\Lambda_{v}\overset{\text{def}}{=}v^{\perp}\cap\Z^{d}.
\]
Note that the map that sends $\Z_{\text{prim }}^{d}\ni v\mapsto\Lambda_{v}$,
is onto $\primlatof{d-1}$. In addition, let
\[
\S^{d-1}(T)\overset{\text{def}}{=}\left\{ v\in\Z_{\text{prim }}^{d}\mid\left\Vert v\right\Vert =T\right\} .
\]
Then, we have the following bijection.
\begin{lem}
\label{lem:bijection orthogonal vectors}The map
\[
\Lambda_{*}:\S^{d-1}(T)\to\primlatof{d-1}(T)
\]
 which sends $v\mapsto\Lambda_{v}$ is a bijection. 
\end{lem}

\begin{proof}
See \cite{AESgrids}, introduction.
\end{proof}
Let $v\in\Z_{\text{prim}}^{d}$ and let $H_{v}\leq\text{SO}_{d}$
be the subgroup stabilizing $v$. We define by $g_{v}\in\text{SL}{}_{d}(\Z)$
to be a matrix who's first $d-1$ columns form a positively oriented
basis for $\Lambda_{v}$. Note that $H_{v}$ and $g_{v}^{-1}H_{v}g_{v}$
are both linear algebraic groups defined over $\Q$, and observe that
$g_{v}^{-1}H_{v}g_{v}\leq\text{ASL}{}_{d-1}$, where
\[
\text{ASL}_{d-1}=\left\{ \left(\begin{array}{cc}
g & *\\
0_{1\times d-1} & 1
\end{array}\right)\mid g\in\text{SL}{}_{d-1}\right\} .
\]
For what follows, we denote by $\Q_{p}$ the field of p-adic numbers
and by $\Z_{p}$ the ring of p-adic integers. Now, recall that $\text{ASL}_{d-1}(\Q_{p})=\text{ASL}_{d-1}(\Z_{p})\text{ASL}_{d-1}\left(\Z\left[\frac{1}{p}\right]\right)$
(see \cite{AESgrids} Section 6.3) and assume that $h\in H_{v}(\Q_{p})\cap\text{SO}{}_{d}(\Z_{p})\text{SO}{}_{d}\left(\Z\left[\frac{1}{p}\right]\right)$.
Then, we may write 

\begin{equation}
h=c_{1}\gamma_{1},\ \ c_{1}\in\text{SO}{}_{d}(\Z_{p}),\ \gamma_{1}\in\text{SO}_{d}\left(\Z\left[\frac{1}{p}\right]\right),\label{eq:rotation decomposition}
\end{equation}

and

\begin{equation}
g_{v}^{-1}hg_{v}=c_{2}\gamma_{2}^{-1},\ c_{2}\in\text{ASL}_{d-1}(\Z_{p}),\ \gamma_{2}\in\text{ASL}_{d-1}\left(\Z\left[\frac{1}{p}\right]\right).\label{eq:lattice decompostion}
\end{equation}

The following lemma, and its corollary, show the principle which is
used to generate elements $\Lambda\in\primlatof{d-1}$ of a fixed
covolume. 
\begin{lem}
\label{lem:genrated vectors and sl_d(Z) matrices}It holds \emph{$\gamma_{1}g_{v}\gamma_{2}\in\text{SL}_{d}(\Z)$}. 
\end{lem}

\begin{proof}
We observe from \eqref{eq:rotation decomposition} and \eqref{eq:lattice decompostion}
that
\[
\text{SL}_{d}\left(\Z\left[\frac{1}{p}\right]\right)\ni\gamma_{1}g_{v}\gamma_{2}=c_{1}g_{v}c_{2}^{-1}\in\text{SL}_{d}(\Z_{p}),
\]
Since $\Z_{p}\cap\Z\left[\frac{1}{p}\right]=\Z$, the statement follows.
\end{proof}
\begin{rem*}
Although not explicitly stated in \cite{AESgrids}, the proof of Lemma
\ref{lem:genrated vectors and sl_d(Z) matrices} can be readily deduced
from the proof of Proposition 6.2 of \cite{AESgrids}.
\end{rem*}
\begin{cor}
\label{cor:generated orthogonal lattices} Let $\Lambda$ be the $\Z$-span
of the first $\left(d-1\right)$ columns of $\gamma_{1}g_{v}\gamma_{2}$.
It holds that $\gamma_{1}v\in\Z_{\text{prim}}^{d}$ and $\Lambda=\Lambda_{\gamma_{1}v}$.
Importantly, \emph{
\[
\text{cov}(\Lambda)=\text{cov}(\Lambda_{v}).
\]
}
\end{cor}

\begin{proof}
Since $\gamma_{1}g_{v}\gamma_{2}\in\text{SL}_{d}(\Z)$, the basis
of $\Lambda$ can be completed to a basis of $\Z^{d}$ which implies
that $\Lambda\in\primlatof{d-1}$. Next, a computation that uses \eqref{eq:covol def}
shows
\begin{equation}
\text{cov}(\Lambda)=\text{cov}(\Lambda_{v}).\label{eq:cov ofz-span of d-1 coloms of gamma_1g_vgamma_2 equal to lambda_v}
\end{equation}
Now observe that $\Lambda\subseteq\gamma_{1}v^{\perp}$. Hence by
Lemma \ref{lem:bijection orthogonal vectors} and \eqref{eq:cov ofz-span of d-1 coloms of gamma_1g_vgamma_2 equal to lambda_v}
we deduce $\gamma_{1}v\in\Z_{\text{prim}}^{d}$ and $\Lambda=\Lambda_{\gamma_{1}v}$.
\end{proof}

\subsection{The S-arithmetic orbits and their projection to the reals}

To ease the notation, we introduce 
\[
\G_{1}\overset{\text{def}}{=}\text{SO}{}_{d},\ \G_{2}\overset{\text{def}}{=}\text{ASL}{}_{d-1},\ \G\overset{\text{def}}{=}\G_{1}\times\G_{2}.
\]

For an odd prime $p$ let
\[
\Y_{p}\overset{\text{def}}{=}\G(\R\times\Q_{p})/\G\left(\Z\left[\frac{1}{p}\right]\right),
\]
where by $\G\left(\Z\left[\frac{1}{p}\right]\right)$ we mean the
diagonal embedding of each $\G_{i}\left(\Z\left[\frac{1}{p}\right]\right)$
factor into $\G_{i}(\R\times\Q_{p}).$ Consider the set 
\[
\U\overset{\text{def}}{=}\G(\R\times\Z_{p})\G\left(\Z\left[\frac{1}{p}\right]\right)\subseteq\Y_{p}.
\]
We now recall the (well known) construction of the projection to the
real coordinate. If 
\[
\left((g_{1,\infty},g_{1,p}),(g_{2,\infty},g_{2,p})\right)\G\left(\Z\left[\frac{1}{p}\right]\right)\in\U,
\]
then we may write for $i\in\{1,2\}$,
\[
g_{i,p}=c_{i,p}\gamma_{i,p},\ \ c_{i,p}\in\G_{i}\left(\Z_{p}\right),\ \gamma_{i,p}\in\G_{i}\left(\Z\left[\frac{1}{p}\right]\right).
\]
Then, the map $q_{\infty}:\U\to\G(\R)/\G(\Z)$ is defined by 
\begin{equation}
q_{\infty}\left(\left((g_{1,\infty},g_{1,p}),\ (g_{2,\infty},g_{2,p})\right)\G\left(\Z\left[\frac{1}{p}\right]\right)\right)\overset{\text{def}}{=}(g_{1,\infty}\gamma_{1,p}^{-1},\ g_{2,\infty}\gamma_{2,p}^{-1})\G\left(\Z\right).\label{eq:q_infty def}
\end{equation}

\subsubsection{The S-arithmetic orbit and its decomposition}

Let $v\in\zprim$ and let $g_{v}$ be as defined in Section \ref{subsec:The-mechanism}.
We define the following diagonal embedding of $H_{v}$,
\[
L_{v}\overset{\text{def}}{=}\left\{ \left(h,g_{v}^{-1}hg_{v}\right)\mid h\in H_{v}\right\} \leq\G.
\]
We choose some $k_{v}\in\text{SO}{}_{d}(\R)$ such that 
\[
k_{v}v=e_{d},
\]
and we denote by $a_{v}$ the diagonal matrix with entries $(\left\Vert v\right\Vert ^{-1/\left(d-1\right)},..,\left\Vert v\right\Vert ^{-1/\left(d-1\right)},\left\Vert v\right\Vert )$.
This choices imply that $a_{v}k_{v}g_{v}\in\text{ASL}_{d-1}(\R)$.
The following orbit is of main importance,
\begin{equation}
O_{v,p}\overset{\text{def}}{=}\left((k_{v},e_{p}),\ (a_{v}k_{v}g_{v},e_{p})\right)\cdot L_{v}(\R\times\Q_{p})\G\left(\Z\left[\frac{1}{p}\right]\right).\label{eq:the s-arithemtic orbit}
\end{equation}
We consider the following decomposition of $H_{v}(\Q_{p})$ into double
cosets 
\begin{equation}
H_{v}(\Q_{p})=\bigsqcup_{h\in M}H_{v}\left(\Z_{p}\right)hH_{v}\left(\Z\left[\frac{1}{p}\right]\right),\label{eq:decomposition of SO_V_perp}
\end{equation}
where $M$ is a set of representatives of the double coset space.
We note that the collection of representatives is finite (see \cite{AESgrids},
section 6.2). We denote 
\[
K\overset{\text{def}}{=}H_{e_{d}}(\R)\cong\text{SO}{}_{d-1}(\R),
\]
and 
\[
\Delta K(\R)\times L_{v}(\Z_{p})\overset{\text{def}}{=}\left\{ \left((k,h),(k,g_{v}^{-1}hg_{v})\right)\mid k\in K,\ h\in H_{v}(\Z_{p})\right\} .
\]
\begin{lem}
\label{lem:decomposition of o_v_p}It holds that 
\begin{equation}
O_{v,p}=\bigsqcup_{h\in M}O_{v,p,h},\label{eq:decomposition _v,p}
\end{equation}
where
\[
O_{v,p,h}=\left(\Delta K\times L_{v}(\Z_{p})\right)\left((k_{v},h),\ (a_{v}k_{v}g_{v},g_{v}^{-1}hg_{v})\right)\G\left(\Z\left[\frac{1}{p}\right]\right).
\]
\end{lem}

\begin{proof}
This follows from a simple computation which uses \eqref{eq:decomposition of SO_V_perp},
and the observation that 
\[
k_{v}H_{v}(\R)k_{v}^{-1}=K.
\]
\end{proof}
Let $q_{1}:\U\to\G_{1}(\Q_{p})$ be the projection to the p-adic coordinate
of $\G_{1}(\R\times\Q_{p})$, and define
\[
M_{0}\overset{\text{def}}{=}\left\{ h\in M\mid h\in q_{1}(\U)\right\} .
\]
We observe that $L_{v}(\Z_{p})(h,g_{v}^{-1}hg_{v})\G\left(\Z\left[\frac{1}{p}\right]\right)$
is either contained in $\U$ or disjoint from it. In particular 
\begin{equation}
L_{v}(\Z_{p})(h,g_{v}^{-1}hg_{v})\G\left(\Z\left[\frac{1}{p}\right]\right)\subseteq\U\iff h\in M_{0}.\label{eq:L_v contained if}
\end{equation}
\begin{cor}
It holds that 
\begin{equation}
O_{v,p}\cap\U=\bigsqcup_{h\in M_{0}}O_{v,p,h}.\label{eq:decomposition of O_v,p cap U}
\end{equation}
\end{cor}

\begin{proof}
This follows from the definition of $\U$, decomposition \eqref{eq:decomposition _v,p}
and observation \eqref{eq:L_v contained if}.
\end{proof}
This allows for the following nice description of $q_{\infty}(\U\cap O_{v,p})$.
\begin{prop}
\label{prop:properties of p-adic projection}For $h\in M_{0}$, choose
$\gamma_{i}(h)\in\G_{i}\left(\Z\left[\frac{1}{p}\right]\right)$,
$i\in\left\{ 1,2\right\} $, by \eqref{eq:rotation decomposition}
and \eqref{eq:lattice decompostion}. The following holds:
\begin{enumerate}
\item For $h\in M_{0}$, 
\begin{equation}
q_{\infty}\left(O_{v,p,h}\right)=\Delta K\left(k_{v}\gamma_{1}^{-1}(h),a_{v}k_{v}g_{v}\gamma_{2}(h)\right)\G(\Z).\label{eq:Image of a fiber of an orbit to the real place}
\end{equation}
\item If $h,h'\in M_{0}$ and $h\neq h^{'}$, then 
\begin{equation}
Kk_{v}\gamma_{1}^{-1}(h)\G_{1}(\Z)\cap Kk_{v}\gamma_{1}^{-1}(h')\G_{1}(\Z)=\emptyset,\label{eq:K cosets empty intersection}
\end{equation}
in particular 
\begin{equation}
q_{\infty}\left(O_{v,p,h}\right)\cap q_{\infty}\left(O_{v,p,h'}\right)=\emptyset.\label{eq:q_infty projection empty intersection}
\end{equation}
\item For $h\in M_{0}$, 
\[
q_{\infty}^{-1}\left(\Delta K\left(k_{v}\gamma_{1}^{-1}(h),a_{v}k_{v}g_{v}\gamma_{2}(h)\right)\G(\Z)\right)\bigcap O_{v,p}=O_{v,p,h}.
\]
\end{enumerate}
\end{prop}

\begin{proof}
For $h\in M_{0}$, we write
\begin{equation}
h=c_{1}(h)\gamma_{1}(h),\ \ g_{v}^{-1}hg_{v}=c_{2}(h)\gamma_{2}^{-1}(h),\label{eq:decomposition of h}
\end{equation}
where $(c_{1}(h),c_{2}(h))\in\G(\Z_{p})$ and $(\gamma_{1}(h),\gamma_{2}(h))\in\G\left(\Z\left[\frac{1}{p}\right]\right).$
\begin{enumerate}
\item Since $\left((\gamma_{1}^{-1}(h),\gamma_{1}^{-1}(h))\ ,(\gamma_{2}(h),\gamma_{2}(h))\right)\in\G\left(\Z\left[\frac{1}{p}\right]\right)$,
we see that 
\begin{align*}
\Delta K\times L_{v}(\Z_{p})\cdot\left((k_{v},h)\ ,(a_{v}k_{v}g_{v},g_{v}^{-1}hg_{v})\right)\G\left(\Z\left[\frac{1}{p}\right]\right) & =
\end{align*}
\[
=\Delta K\times L_{v}(\Z_{p})\left((k_{v}\gamma_{1}^{-1}(h),c_{1}(h)),\ (a_{v}k_{v}g_{v}\gamma_{2}(h),c_{2}(h))\right)\G\left(\Z\left[\frac{1}{p}\right]\right).
\]
Hence by definition \eqref{eq:q_infty def},
\[
q_{\infty}\left(O_{v,p,h}\right)=\Delta K\left(k_{v}\gamma_{1}^{-1}(h),a_{v}k_{v}g_{v}\gamma_{2}(h)\right)\G(\Z).
\]
\item The proof of \eqref{eq:K cosets empty intersection} is a routine
check, hence we omit its details and leave them for the reader (one
may also look at the proof of Proposition 6.2 in \cite{AESgrids}).
Note that \eqref{eq:q_infty projection empty intersection} follows
from \eqref{eq:K cosets empty intersection}.
\item This fact follows immediately from the two last ones and \eqref{eq:decomposition of O_v,p cap U}.
\end{enumerate}
\end{proof}

\subsection{\label{subsec:Equivalence-class-of-lattices}The resulting elements
of $\protect\primlatof{d-1}$}

The following commuting diagram will be important for us,
\begin{equation}
\xymatrix{\U\ar^{q_{\infty\ \ \ \ \ \ \ }}[r] & \G(\R)\diagup\G(\Z)\ar[r]_{q_{\Delta K\ \ \ \ \ }} & \Delta K\diagdown\G(\R)\diagup\G(\Z)\\
 & \G(\R)\diagup\G_{2}(\Z)\ar_{id\times q_{P\diagup Q}^{\G_{2}(\R)\diagup\G_{2}(\Z)}}[d]\ar[u]^{\pi_{\G_{1}(\Z)}}\ar[r]_{\tilde{q}_{\Delta K}\ \ \ \ \ } & \Delta K\diagdown\G(\R)\diagup\G_{2}(\Z)\ar_{\tilde{q}}[d]\ar[u]^{\tilde{\pi}_{\G_{1}(\Z)}}\ar[dr]^{\tilde{\mult}}\\
 & \G_{1}(\R)\times P\diagup Q\ar[r]_{q_{\Delta K^{\pm}}} & X_{d-1,d}^{\text{polar}}\ar[r]_{\mult} & X_{d-1,d}
}
\label{eq:main diagram}
\end{equation}

The maps $\pi_{\G_{1}(\Z)}$ and $\tilde{\pi}_{\G_{1}(\Z)}$ are obtained
by dividing from the right by $\G_{1}(\Z)$. The maps $q_{\Delta K}$,
$\tilde{q}_{\Delta K}$ and $q_{\Delta K^{\pm}}$ are obtained by
dividing from the left by $\Delta K$ and $\Delta K^{\pm}$ correspondingly.
The map 
\[
q_{P\diagup Q}^{\G_{2}(\R)\diagup\G_{2}(\Z)}:\G_{2}(\R)\diagup\G_{2}(\Z)\to P\diagup Q,
\]
is naturally defined by $\pi_{2}\left(g\G_{2}(\Z)\right)=gQ$, since
$\G_{2}(\Z)\leq Q$. The maps $\tilde{q}$ and $\tilde{\mult}$ are
defined so that the diagrams commute. Now, denote 
\begin{equation}
\tilde{R}_{v}\overset{\text{def}}{=}\tilde{q}_{\Delta K}\left(\pi_{\G_{1}(\Z)}^{-1}\left(q_{\infty}(\U\cap O_{v,p})\right)\right),\label{eq:R^tilde_v}
\end{equation}
then, by Proposition \ref{prop:properties of p-adic projection},
we get that $\tilde{R}_{v}$ is the finite collection of points 
\[
\tilde{O}_{v,h,\gamma}\overset{\text{def}}{=}\Delta K\left(k_{v}\gamma_{1}^{-1}(h)\gamma,a_{v}k_{v}g_{v}\gamma_{2}(h)\G_{2}(\Z)\right),\ \ \gamma\in\G_{1}(\Z),\ h\in M_{0},
\]
where $\gamma_{i}(h)$, $i\in\left\{ 1,2\right\} $, are defined in
\eqref{eq:decomposition of h}. Denote 
\[
\mathcal{L}(v)\overset{\text{def}}{=}\tilde{\mult}(\tilde{R}_{v})\subseteq X_{d-1,d}.
\]
\begin{lem}
\label{lem:elements of L_v}It holds that 
\[
\mathcal{L}(v)=\left\{ \left[\Lambda(v,h,\gamma)\right]\overset{\text{def}}{=}\gamma^{-1}\gamma_{1}(h)g_{v}\gamma_{2}(h)Q\mid h\in M_{0},\ \gamma\in\G_{1}(\Z)\right\} .
\]
Importantly $\left[\Lambda(v,h,\gamma)\right]=\left[\Lambda_{\gamma^{-1}\gamma_{1}(h)v}\right]$
and as a consequence $\mathcal{L}(v)\subseteq\primlatof{d-1}(\left\Vert v\right\Vert )$.
\end{lem}

\begin{proof}
We have 
\[
\tilde{\mult}\left(\tilde{O}_{v,h,\gamma}\right)=\gamma^{-1}\gamma_{1}(h)k_{v}^{-1}\left(a_{v}k_{v}g_{v}\gamma_{2}(h)Q\right),
\]
and we note that $a_{v}k_{v}g_{v}\gamma_{2}(h)Q=k_{v}g_{v}\gamma_{2}(h)Q$,
which gives
\[
\tilde{\mult}\left(\tilde{O}_{v,h,\gamma}\right)=\gamma^{-1}\gamma_{1}(h)\left(k_{v}^{-1}k_{v}\right)g_{v}\gamma_{2}(h)Q=\left[\Lambda(v,h,\gamma)\right].
\]
By Corollary \ref{cor:generated orthogonal lattices}, 
\begin{equation}
\left[\Lambda(v,h,\gamma)\right]=\left[\Lambda_{\gamma^{-1}\gamma_{1}(h)v}\right],\label{eq:the resulting primitive lattice as orthogonal latice}
\end{equation}
and also $\mathcal{L}(v)\subseteq\primlatof{d-1}(\left\Vert v\right\Vert )$.
\end{proof}

\subsubsection{Refinement of Theorem \ref{thm:maintheorem}}

For everything that follows we fix a prime $p\neq2$ and assume that
$d\geq4$ is a natural number.
\begin{defn}
We shall say that $v\in\Z_{\text{prim}}^{d}$ is admissible, if either
of the following holds
\begin{enumerate}
\item $d=4$, and $\left\Vert v\right\Vert ^{2}\subseteq\mathbb{D}(p)/8\N$.
\item $d=5$, and $\left\Vert v\right\Vert ^{2}\subseteq\mathbb{D}(p)$.
\item $d>5$, and $v$ is any primitive vector.
\end{enumerate}
\end{defn}

In section \ref{sec:proof of the equidisitrbution} we shall conclude
the following theorem.
\begin{thm}
\label{thm:refinement of main theorem} Let $\left\{ v_{i}\right\} _{i=1}^{\infty}$
be a sequence of admissible vectors such that 
\[
\left\Vert v_{i}\right\Vert \to\infty,
\]
 and let $\mu_{v_{i}}$ be the uniform counting measures supported
on $\mathcal{L}(v_{i})$. Then \emph{
\[
\mu_{v_{i}}\overset{\text{weak *}}{\longrightarrow}\mu_{\text{polar}}.
\]
}
\end{thm}

Theorem \ref{thm:refinement of main theorem} implies Theorem \ref{thm:maintheorem}
by the following. In \cite{AESgrids} (see Section 5.1 and Proposition
6.2 in \cite{AESgrids}) there was introduced an equivalence relation
on the primitive vectors lying on spheres. It was shown that the equivalence
class of $v\in\Z_{\text{prim}}^{d}$ is exactly 
\[
\left\{ \gamma^{-1}\gamma_{1}(h)v\right\} _{\gamma\in\G_{1}(\Z),\ h\in M_{0}}.
\]
Hence by Lemma \ref{eq:decomposition of h}, if $v\sim u$ then $\mathcal{L}(v)=\mathcal{L}(u)$. 

\section{\label{sec:proof of the equidisitrbution}The resulting measures}

\subsection{A further refinement}

We define $\tilde{\nu}_{v}$ be the uniform measure on the finite
set $\tilde{R}_{v}$ (defined in \eqref{eq:R^tilde_v}). We also define
the measure 
\begin{equation}
\tilde{\nu}_{\text{polar}}=\left(\tilde{q}_{\Delta K}\right)_{*}\mu_{\G_{1}(\R)}\otimes\mu_{\G_{2}(\R)\diagup\G_{2}(\Z)},\label{eq:nu_tilde_polar}
\end{equation}
 where $\mu_{\G_{1}(\R)}$ and $\mu_{\G_{2}(\R)\diagup\G_{2}(\Z)}$
are the Haar probability measure on $\G_{1}(\R)$ and the $\G_{2}(\R)$-invariant
probability measure on $\G_{2}(\R)\diagup\G_{2}(\Z)$. We will prove
the following.
\begin{thm}
\label{thm:convergence of nu tilde-1}Let $\left\{ v_{i}\right\} _{i=1}^{\infty}$
be a sequence of admissible vectors such that 
\[
\left\Vert v_{i}\right\Vert \to\infty,
\]
 then\emph{
\begin{equation}
\tilde{\nu}_{v_{i}}\overset{\text{weak *}}{\longrightarrow}\tilde{\nu}_{\text{polar}}.\label{eq:nu_tilde convergence to natural measure}
\end{equation}
}
\end{thm}

Theorem \ref{thm:convergence of nu tilde-1} implies Theorem \ref{thm:refinement of main theorem}
by the following two lemmata.
\begin{lem}
\label{lem:push of counting measure nu_tilde_v}It holds that the
map $\tilde{\mult}$ when restricted to $\tilde{R}_{v}$ is a bijection
onto $\mathcal{L}(v)$. In particular, 
\begin{equation}
\tilde{\mult}_{*}\tilde{\nu}_{v}=\mu_{v}.\label{eq:mult_tild_nu_tild_v is mu_v}
\end{equation}
\end{lem}

\begin{proof}
The map is clearly onto. In order to prove injectivity, we recall
that part 2 of Proposition \ref{prop:properties of p-adic projection}
states that
\[
Kk_{v}\gamma_{1}^{-1}(h)\G_{1}(\Z)\cap Kk_{v}\gamma_{1}^{-1}(h')\G_{1}(\Z)=\emptyset,\ \ h\neq h',\ h,h'\in M_{0},
\]
which implies that for different representatives $h,h'\in M_{0}$,
the corresponding $\left(d-1\right)$-subgroups defined by \eqref{eq:the resulting primitive lattice as orthogonal latice}
lie inside different hyperplanes. Finally, since bijectivity is established,
we immediately get \eqref{eq:mult_tild_nu_tild_v is mu_v}.
\end{proof}
\begin{lem}
\label{lem:push of nu_tilde_v_polar}It holds that $\tilde{\mult}_{*}\tilde{\nu}_{\text{polar}}=\mu_{\text{polar }}$.
\end{lem}

\begin{proof}
Since $\tilde{\mult}=\mult\circ\tilde{q}$ and since $\mult_{*}\nu_{\text{polar}}=\mu_{\text{polar }}$,
it is left to prove $\tilde{q}_{*}\tilde{\nu}_{\text{polar}}=\nu_{\text{polar}}.$
We note that 
\[
\left(q_{P\diagup Q}^{\G_{2}(\R)\diagup\G_{2}(\Z)}\right)_{*}\mu_{\G_{2}(\R)\diagup\G_{2}(\Z)}=\mu_{X_{d-1}},
\]
and observe by Diagram \eqref{eq:main diagram} that
\begin{equation}
\tilde{q}\circ\tilde{q}_{\Delta K}=q_{\Delta K^{\pm}}\circ\left(id\times q_{P\diagup Q}^{\G_{2}(\R)\diagup\G_{2}(\Z)}\right).\label{eq:q_tilde equalities}
\end{equation}
Therefore,
\[
\tilde{q}_{*}\tilde{\nu}_{\text{polar}}\underbrace{=}_{\text{definition of \ensuremath{\tilde{\nu}_{\text{polar}}}}}\left(\tilde{q}\right)_{*}\left(\left(\tilde{q}_{\Delta K}\right)_{*}\mu_{\G_{1}(\R)}\otimes\mu_{\G_{2}(\R)\diagup\G_{2}(\Z)}\right)\underbrace{=}_{\eqref{eq:q_tilde equalities}}
\]
\[
\left(q_{\Delta K^{\pm}}\right)_{*}\left(id\times q_{P\diagup Q}^{\G_{2}(\R)\diagup\G_{2}(\Z)}\right)_{*}\mu_{\G_{1}(\R)}\otimes\mu_{\G_{2}(\R)\diagup\G_{2}(\Z)}=
\]
\[
\left(q_{\Delta K^{\pm}}\right)_{*}\mu_{\G_{1}(\R)}\otimes\mu_{X_{d-1}}\underbrace{=}_{\text{\eqref{eq:nu_polar definition}}}\nu_{\text{polar }}.
\]
\end{proof}

\subsection{Proof of Theorem \ref{thm:convergence of nu tilde-1}}

To summarize, we have 
\[
\text{Theorem \ref{thm:convergence of nu tilde-1} \ensuremath{\implies}Theorem \ref{thm:refinement of main theorem} \ensuremath{\implies}Theorem \ref{thm:maintheorem}}.
\]
Hence this section serves as the last step of the proof for Theorem
$\ref{thm:maintheorem}.$

\subsubsection{The key Theorem of \cite{AESgrids}}

The orbit $O_{v,p}$ defined in \eqref{eq:the s-arithemtic orbit}
is a compact orbit (see \cite{AESgrids}, Section 3.2) and we denote
by $\mu_{O_{v,p}}$ the $L_{v}(\R\times\Q_{p})$-invariant probability
measure supported on $O_{v,p}$. Also, let $\mu_{\Y_{p}}$ be the
$\G(\R\times\Q_{p})$-invariant probability measure on $\Y_{p}$.
The following theorem, which was proved in \cite{AESgrids}, is key
in order to prove Theorem \ref{thm:convergence of nu tilde-1}.
\begin{thm}
\emph{\label{thm:AESgrids thm} }Let $\left\{ v_{i}\right\} _{i=1}^{\infty}$
be a sequence of admissible vectors such that 
\[
\left\Vert v_{i}\right\Vert \to\infty,
\]
then 
\[
\mu_{O_{v_{i},p}}\overset{\text{weak * }}{\longrightarrow}\mu_{\Y_{p}}.
\]
\end{thm}

We define the probability measure $\eta_{v}$ on $O_{v,p}\cap\U$,
by
\begin{equation}
\eta_{v}\overset{\text{def}}{=}\mu_{O_{v,p}}\mid_{\U}.\label{eq:eta_v definition}
\end{equation}
 Since $\U$ is a clopen set, it follows from Theorem \ref{thm:AESgrids thm}
that
\[
\eta_{v_{i}}\overset{\text{weak *}}{\longrightarrow}\mu_{\Y_{p}}\mid_{\U}.
\]
Also, since $q_{\infty}$ is a proper map we get 
\[
\left(q_{\infty}\right)_{*}\eta_{v_{i}}\overset{\text{weak *}}{\longrightarrow}\left(q_{\infty}\right)_{*}\mu_{\Y_{p}}\mid_{\U}.
\]
Importantly, $\mu_{\Y_{p}}\mid_{\U}$ is $\G(\R)$ invariant, hence
also $\left(q_{\infty}\right)_{*}\mu_{\Y_{p}}\mid_{\U}$. Therefore
we deduce,
\begin{cor}
\label{cor:s-arithemetic measures converge to haar}It holds that
\begin{equation}
\left(q_{\infty}\right)_{*}\eta_{v_{i}}\overset{\text{weak *}}{\longrightarrow}\mu_{\G(\R)/\G(\Z)},\label{eq:weak convergence to haar of eta}
\end{equation}
where $\mu_{\G(\R)/\G(\Z)}$ is the $\G(\R)$-invariant probability
on $\G(\R)/\G(\Z)$.
\end{cor}

Next, note that Proposition \ref{prop:properties of p-adic projection}
shows that the measure $\left(q_{\infty}\right)_{*}\eta_{v}$ is supported
on a finite union of $\Delta K$ orbits
\[
\bigsqcup_{h\in M_{0}}q_{\infty}\left(O_{v,p,h}\right),
\]
and by applying further $q_{\Delta K}$, we get that $\left(q_{\Delta K}\circ q_{\infty}\right)_{*}\eta_{v}$
is supported on a finite set 
\[
R_{v}\overset{\text{def}}{=}q_{\Delta K}\circ q_{\infty}(O_{v,p}),
\]
which consists of the elements
\[
\tilde{O}_{v,h}=\Delta K\left(k_{v}\gamma_{1}^{-1}(h),a_{v}k_{v}g_{v}\gamma_{2}(h)\right)\G(\Z),\ h\in M_{0}.
\]
On the other hand, note that
\[
\tilde{\pi}_{\G_{1}(\Z)}\left(\tilde{R}_{v}\right)=R_{v},
\]
so that $\left(\tilde{\pi}_{\G_{1}(\Z)}\right)_{*}\tilde{\nu}_{v}$
has the same support as that of $\left(q_{\Delta K}\circ q_{\infty}\right)_{*}\eta_{v}$.
The following lemma connects those two measures.
\begin{lem}
\label{lem:difference of measure is zero}It holds that 
\[
\left(\pi_{\G_{1}(\Z)}\right)_{*}\tilde{\nu}_{v_{i}}-\left(q_{\Delta K}\circ q_{\infty}\right)_{*}\eta_{v_{i}}\overset{\text{weak *}}{\longrightarrow}0.
\]
\end{lem}

In Subsection \ref{subsec:lemma of difference of meas} we will explain
how Lemma \ref{lem:difference of measure is zero} follows from \cite{AESgrids}.
Before that, we explain how Lemma \ref{lem:difference of measure is zero}
and the preceding discussion implies Theorem \ref{thm:convergence of nu tilde-1}. 
\begin{proof}[Proof of Theorem \ref{thm:convergence of nu tilde-1}]
 By Corollary \ref{cor:s-arithemetic measures converge to haar},
it follows that 
\[
\left(q_{\Delta K}\circ q_{\infty}\right)_{*}\eta_{v_{i}}=\left(q_{\Delta K}\right)_{*}\left((q_{\infty})_{*}\eta_{v_{i}}\right)\overset{\text{weak *}}{\longrightarrow}\left(q_{\Delta K}\right)_{*}\mu_{\G(\R)\diagup\G(\Z)}.
\]
Hence we get from Lemma \ref{lem:difference of measure is zero} that
\[
\left(\pi_{\G_{1}(\Z)}\right)_{*}\tilde{\nu}_{v_{i}}\overset{\text{weak *}}{\longrightarrow}\left(q_{\Delta K}\right)_{*}\mu_{\G(\R)\diagup\G(\Z)}.
\]
Observe that (see Diagram \eqref{eq:main diagram})
\[
\left(q_{\Delta K}\right)_{*}\mu_{\G(\R)\diagup\G(\Z)}=\left(\pi_{\G_{1}(\Z)}\right)_{*}\left(\left(\tilde{q}_{\Delta K}\right)_{*}\mu_{\G_{1}(\R)}\otimes\mu_{\G_{2}(\R)\diagup\G_{2}(\Z)}\right),
\]
so that by the definition of $\tilde{\nu}_{\text{polar}}$ (see \eqref{eq:nu_tilde_polar}),
we get that
\[
\left(\pi_{\G_{1}(\Z)}\right)_{*}\tilde{\nu}_{v_{i}}\overset{\text{weak *}}{\longrightarrow}\left(\pi_{\G_{1}(\Z)}\right)_{*}\tilde{\nu}_{\text{polar}}.
\]
Now, since the measures $\tilde{\nu}_{v}$ and $\tilde{\nu}_{\text{polar}}$
are both $\G_{1}(\Z)$ invariant and since $\G_{1}(\Z)$ is finite,
we also obtain that 
\[
\tilde{\nu}_{v_{i}}\overset{\text{weak *}}{\longrightarrow}\tilde{\nu}_{\text{polar}}.
\]
\end{proof}

\subsubsection{\label{subsec:lemma of difference of meas}Outline of the proof for
Lemma \ref{lem:difference of measure is zero}}

Let $\lambda_{v}$ be the uniform counting measures on the sets $R_{v}.$
The idea of the proof of Lemma \ref{lem:difference of measure is zero}
is to show that 
\begin{equation}
\left(\pi_{\G_{1}(\Z)}\right)_{*}\tilde{\nu}_{v}-\lambda_{v}\to0,\ \text{and\  }\left(q_{\Delta K}\circ q_{\infty}\right)_{*}\eta_{v}-\lambda_{v}\to0.\label{eq:difference of counting measures}
\end{equation}
We denote 
\begin{equation}
\left(q_{\Delta K}\circ q_{\infty}\right)_{*}\eta_{v}=\sum_{h\in M_{0}}\alpha_{v,h}\delta_{\tilde{O}_{v,h}},\label{eq:rh_Dk_rh_infty_eta}
\end{equation}
then, 
\[
\alpha_{v,h}\overset{\text{def}}{=}\eta_{v}(q_{\infty}^{-1}(O_{v,h}))\underbrace{=}_{\eqref{eq:eta_v definition}\text{, and Proposition \ref{prop:properties of p-adic projection}}}\eta_{v}(O_{v,p,h}).
\]
It follows that 
\[
\eta_{v}(O_{v,p,h})=\frac{\alpha}{\left|\text{stab}_{\Delta K\times L_{v}(\Z_{p})}(k_{v},h,a_{v}k_{v}g_{v},g_{v}^{-1}hg_{v})\G(\Z(\frac{1}{p}))\right|},
\]
where $\alpha=\alpha(v)$ normalizes $\left(q_{\Delta K}\circ q_{\infty}\right)_{*}\eta_{v}$
to be a probability measure. Also, let
\[
\left(\pi_{\G_{1}(\Z)}\right)_{*}\tilde{\nu}_{v}=\sum_{h\in M_{0}}\beta_{v,h}\delta_{\tilde{O}_{v,h}},
\]
where
\[
\beta_{v,h}=\frac{\beta}{\left|\text{stab}_{\G_{1}(\Z)}(Kk_{v}\gamma_{1}^{-1}(h))\right|},
\]
and $\beta=\beta(v)$ normalizes the measure $\left(\pi_{\G_{1}(\Z)}\right)_{*}\tilde{\nu}_{v}$
to a probability measure. Let 
\[
M_{v}=\max_{h\in M_{0}}\alpha_{v,h},
\]
and 
\[
N_{v}=\max_{h\in M_{0}}\beta_{v,h}.
\]

Also let 
\[
E=\left\{ \Delta K\left(\rho,\eta)\G(\Z)\right)\mid\left|\text{stab}_{\G_{1}(\Z)}(K\rho)\right|>1\right\} .
\]
The following statements were proven in \cite{AESgrids}, 
\begin{lem}
\label{lem:weight lemma from aes-1}The following holds,
\begin{enumerate}
\item For all $h\in M_{0}$ such that $O_{v,h}\notin E,$ it holds that
$\alpha_{v,h}=M_{v}$ and $\beta_{v,h}=N_{v}$.
\item $\frac{\left|R_{v}\cap E\right|}{|R_{v}|}\to0$. 
\end{enumerate}
\end{lem}

\begin{proof}
See Lemmata 6.3 and 6.4 of \cite{AESgrids}.
\end{proof}
It is immediate that Lemma \ref{lem:weight lemma from aes-1} implies
the limits \eqref{eq:difference of counting measures}.

\subsubsection*{Acknowledgements}

I thank Uri Shapira for purposing this problem, his valuable support
and for many discussions. I also thank Cheng Zheng and Rene R{\"u}hr  for
many important discussions on this project.

\end{document}